\documentclass[reqno,11pt]{amsart}
\usepackage[margin=1in]{geometry}

\usepackage{amsthm, amsmath, amssymb, asymptote, bm}
\usepackage{tikz}
\usepackage{thmtools}
\usepackage{thm-restate}
\usepackage[utf8]{inputenc}
\usepackage{pgfplots}

\usepackage{microtype} 

\usepackage[utf8]{inputenc}
\usepackage[T1]{fontenc}

\usepackage[textsize=small,backgroundcolor=orange!20]{todonotes}

\usepackage[hidelinks]{hyperref}
\usepackage{url}

\usepackage{xcolor,graphics}

\newtheorem{theorem}{Theorem}[section]

\newtheorem{lemma}[theorem]{Lemma}

\newtheorem{conjecture}[theorem]{Conjecture}
\newtheorem*{question*}{Question}

\theoremstyle{definition}
\newtheorem{definition}[theorem]{Definition}

\newtheorem{case}{Case}
\newtheorem*{remark}{Remark}

\theoremstyle{remark}

\newcommand{\wt}{\widetilde}

\newcommand{\Beta}{\mathrm{B}}

\tikzstyle{P} = [draw, circle, black, fill, inner sep = 0pt, minimum width = 3pt]
\tikzstyle{every loop} = []

\title{Constraining Strong c-Wilf Equivalence Using Cluster Poset Asymptotics}

\author[Mitchell Lee]{Mitchell Lee}
\address{Harvard University, Cambridge, MA 02138, USA}
\email{mitchell@math.harvard.edu}
\author[Ashwin Sah]{Ashwin Sah}
\address{Massachusetts Institute of Technology, Cambridge, MA 02139, USA}
\email{asah@mit.edu}

\date{May 2018}

\begin{document}

\begin{abstract}
Let $\pi \in \mathfrak{S}_m$ and $\sigma \in \mathfrak{S}_n$ be permutations. An \emph{occurrence} of $\pi$ in $\sigma$ as a consecutive pattern is a subsequence $\sigma_i \sigma_{i+1} \cdots \sigma_{i+m-1}$ of $\sigma$ with the same order relations as $\pi$. We say that patterns $\pi, \tau \in \mathfrak{S}_m$ are \emph{strongly c-Wilf equivalent} if for all $n$ and $k$, the number of permutations in $\mathfrak{S}_n$ with exactly $k$ occurrences of $\pi$ as a consecutive pattern is the same as for $\tau$. In 2018, Dwyer and Elizalde \cite{dwyer2018wilf} conjectured (generalizing a conjecture of Elizalde \cite{elizalde2012most} from 2012) that if $\pi, \tau \in \mathfrak{S}_m$ are strongly c-Wilf equivalent, then $(\tau_1, \tau_m)$ is equal to one of $(\pi_1, \pi_m)$, $(\pi_m, \pi_1)$, $(m+1 - \pi_1, m+1-\pi_m)$, or $(m+1 - \pi_m, m+1 - \pi_1)$. We prove this conjecture using the cluster method introduced by Goulden and Jackson in 1979 \cite{goulden1979inversion}, which Dwyer and Elizalde previously applied \cite{dwyer2018wilf} to prove that $|\pi_1 - \pi_m| = |\tau_1 - \tau_m|$. A consequence of our result is the full classification of c-Wilf equivalence for a special class of permutations, the non-overlapping permutations. Our approach uses analytic methods to approximate the number of linear extensions of the ``cluster posets'' of Elizalde and Noy \cite{elizalde2012clusters}.
\end{abstract}

\maketitle

\section{Introduction}\label{sec:introduction}
Permutation patterns have held considerable interest ever since their introduction by Knuth in 1968 \cite{knuth1968art}. Of particular importance is the number of permutations in $\mathfrak{S}_n$ avoiding a fixed pattern $\pi$. Variations on this problem involve adding restrictions to the pattern: for instance, requiring that it appear in consecutive positions. This notion of \emph{consecutive pattern avoidance} was first systematically studied by Elizalde and Noy in 2003 \cite{elizalde2003consecutive} and has been considered from multiple viewpoints \cite{elizalde2006asymptotic, elizalde2012most, nakamura2480computational}. For more information on the consecutive pattern avoidance literature, we refer the reader to the 2016 survey of Elizalde \cite{elizalde2016survey}.

\par We now recall the definition of consecutive pattern avoidance. Given a sequence $\tau$ of $m$ distinct positive integers, define the \emph{standardization} of $\tau$ to be the sequence formed by replacing the $i$th smallest entry of $\tau$ by $i$ for $1 \leq i \leq m$. Given permutations $\pi \in \mathfrak{S}_m$ and $\sigma \in \mathfrak{S}_n$, an \emph{occurrence} of $\pi$ in $\sigma$ as a consecutive pattern is a subsequence $\sigma_i \sigma_{i+1} \cdots \sigma_{i+m-1}$ of $\sigma$ whose standardization is $\pi$. We say that $\sigma$ \emph{avoids} $\pi$ as a consecutive pattern if it has no occurrence of $\pi$ as a consecutive pattern.

\par We say that $\pi, \tau \in \mathfrak{S}_m$ are \emph{c-Wilf equivalent} and write $\pi \sim \tau$ if for all $n$, the number of permutations $\sigma \in \mathfrak{S}_n$ that avoid $\pi$ (as a consecutive pattern) is the same as the number that avoid $\tau$. We say that $\pi, \tau \in \mathfrak{S}_m$ are \emph{strongly c-Wilf equivalent} and write $\pi \overset{s}{\sim} \tau$ if for all $n$ and $k$, the number of permutations $\sigma \in \mathfrak{S}_n$ with $k$ occurrences of $\pi$ is the same as the number with $k$ occurrences of $\tau$. Clearly, if $\pi \overset{s}{\sim} \tau$ then $\pi \sim \tau$. Nakamura conjectured in 2011 that the reverse implication also holds \cite[Conjecture~6]{nakamura2480computational}.

\par A trivial example of strong c-Wilf equivalence is as follows. For any permutation $\pi = \pi_1\cdots\pi_m$, define the \emph{reverse} $\pi^R = \pi_m\cdots\pi_1$ and the \emph{complement} $\pi^C = (m+1-\pi_1)\cdots (m+1-\pi_m)$. Then $\pi\overset{s}\sim\pi^R\overset{s}\sim\pi^C\overset{s}\sim\pi^{RC}$. However, there are other strong c-Wilf equivalences; for example, $1342\overset{s}\sim 1432$. A comprehensive table of c-Wilf equivalence classes for permutations of length at most $5$ can be found in \cite[Table~1]{dwyer2018wilf}.

Following Dwyer and Elizalde \cite{dwyer2018wilf}, we say a permutation $\pi \in \mathfrak{S}_m$ is \emph{standard} if $\pi_1 < \pi_m$ and $\pi_1 + \pi_m\le m + 1$. At least one of the permutations $\pi$, $\pi^R$, $\pi^C$, and $\pi^{RC}$ is standard, so it suffices to study (strong) c-Wilf equivalence of standard permutations. Still, the problem of determining whether two permutations are (strongly) c-Wilf equivalent seems difficult in general.

We complete the solution to this problem for a certain restricted class of permutations. Following B\'ona \cite{bona2011non}, we say that $\pi \in \mathfrak{S}_m$ is \emph{non-overlapping} if $\pi_1 \cdots \pi_i$ and $\pi_{m-i + 1} \cdots \pi_m$ have different standardizations for $2 \leq i \leq m-1$. B{\'o}na showed that the fraction of non-overlapping permutations approaches $0.36409\ldots$. Furthermore, in 2011 Duane and Remmel \cite{duane2011minimal} showed that if $\pi, \tau \in \mathfrak{S}_m$ are non-overlapping and $(\pi_1, \pi_m) = (\tau_1, \tau_m)$, then $\pi\overset{s}\sim\tau$.

In 2012, Elizalde conjectured a converse to this result, which we prove in this paper.

\begin{restatable}[conjectured in {\cite[p.14]{elizalde2012most}}]{theorem}{thmnons} \label{thm:non-s}
Let $\pi, \tau \in \mathfrak{S}_m$ be non-overlapping, standard permutations. If $\pi\sim\tau$, then $(\pi_1, \pi_m) = (\tau_1, \tau_m)$.
\end{restatable}

This completes the classification of c-Wilf equivalence for non-overlapping standard permutations, and hence all non-overlapping permutations.

As a consequence of Theorem~\ref{thm:non-s} and \cite[Theorem~8]{dwyer2018wilf}, we have the following statement about strong c-Wilf equivalence of \emph{all} permutations.

\begin{restatable}[{\cite[Conjecture~7]{dwyer2018wilf}}]{corollary}{corstrong} \label{cor:strong}
Let $\pi, \tau \in \mathfrak{S}_m$ be standard permutations. If $\pi\overset{s}\sim\tau$, then $(\pi_1, \pi_m) = (\tau_1, \tau_m)$.
\end{restatable}

Our proof of Theorem~\ref{thm:non-s} uses the \emph{cluster method} introduced by Goulden and Jackson in 1979 \cite{goulden1979inversion}, which Dwyer and Elizalde previously applied to show that if $\pi, \tau \in \mathfrak{S}_m$ are standard and $\pi \overset{s}{\sim} \tau$, then $\pi_1 - \pi_m = \tau_1 - \tau_m$ \cite[Theorem~9]{dwyer2018wilf}.

The cluster method applies to Theorem~\ref{thm:non-s} as follows. In 2012, Elizalde and Noy \cite{elizalde2012clusters} defined the \emph{cluster posets} associated to a permutation $\pi \in \mathfrak{S}_m$. If $\pi$ is non-overlapping, then it has only one cluster poset $P_n^{\pi}$. This poset depends only on $m$, $n$, $a = \pi_1$, and $b = \pi_m$, so we will often write it as $P_n^{m, a, b}$. It can be explicitly described as follows.

\begin{definition}[cf. {\cite[Section~2.2]{dwyer2018wilf}}]
Let $a$, $b$, and $m$ be integers with $1 \leq a < b \leq m$. For any $n \geq 1$, define the cluster poset $P_n^{m,a,b}$ by gluing together $n$ chains of length $m$ as follows. The elements of $P_n^{m, a, b}$ are $A_{i, j}$ for $1 \leq i \leq n$ and $1 \leq j \leq m$, where $A_{i, b} = A_{i+1, a}$ for $1 \leq i \leq n-1$ and the $A_{i, j}$ are distinct otherwise (so $(m-1)n+1$ elements in total). The relations of $P_n^{m, a, b}$ are generated by the relations $A_{i, 1} \leq \cdots \leq A_{i, m}$ for $1 \leq i \leq n$.
\end{definition}
The bottommost chain corresponds to $A_{1, j}$. An example is shown for $m = 8, a = 3, b = 5$ in Figure~\ref{fig:cluster-poset}.

\begin{figure}
\includegraphics{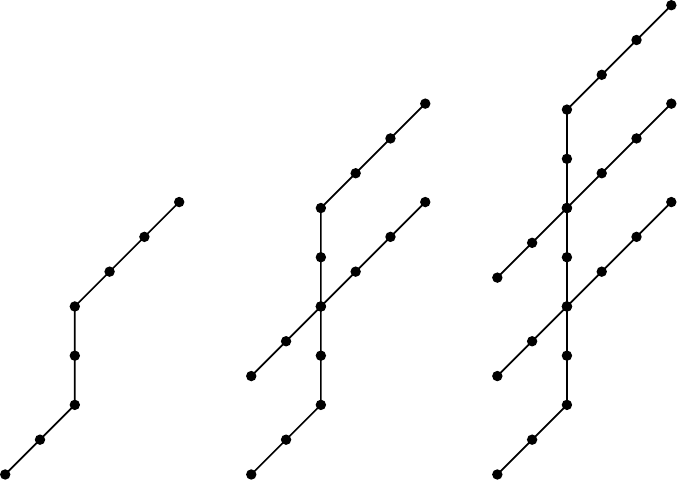}
\caption{The Hasse diagrams of the cluster posets $P_1^{8,3,5}, P_2^{8,3,5}, P_3^{8,3,5}$.}
\label{fig:cluster-poset}
\end{figure}

Dwyer and Elizalde \cite[Theorem~14]{dwyer2018wilf} showed that if $\pi, \tau \in \mathfrak{S}_m$ are non-overlapping and $\pi \sim \tau$, then the posets $P^\pi_n$ and $P^\tau_n$ have the same number of linear extensions for all $n$. Hence, to prove Theorem~\ref{thm:non-s}, it suffices to approximate the number of linear extensions of $P_n^{m, a, b}$ accurately enough to determine the pair $(a, b)$. Dwyer and Elizalde accomplished this for $b - a = 1$ \cite[Proposition~20]{dwyer2018wilf}, and the case $b - a > 1$ is completed by the following theorem and its corollary. We use the notation $e(P)$ for the number of linear extensions of a finite poset $P$.

\begin{theorem}\label{thm:main}
Fix integers $a$, $b$, and $m$ with $1 \leq a < b \leq m$. Then
\[\log e(P_n^{m,a,b}) = (m-b+a-1)n\log n + c(m,a,b)n + O_{m,a,b}(\log n),\]
where
\begin{align*}
c(m, a, b) =&~(b - a)\log \Beta \left(\frac{a - 1}{b - a}+1, \frac{m - b}{b - a}+1\right) - \log \Beta (a, m-b+1) \\
&- \log\Gamma(m - b + a + 1) + (m - 1)\log (m - 1) - (b - a)\log(b - a) - m + b - a + 1.
\end{align*}
Here $\Beta$ is the beta function, given by \[\Beta(\alpha, \beta) = \int_{0}^1 u^{\alpha - 1} (1 - u)^{\beta - 1} du = \frac{\Gamma(\alpha) \Gamma(\beta)}{\Gamma(\alpha + \beta)}.\]
\end{theorem}

\begin{remark}
The weaker statement that $\log e(P_n^{m, a, b}) = (m-b+a-1)n \log n + O_{m, a, b}(n)$ is \cite[Lemma~18]{dwyer2018wilf}.
\end{remark}

\begin{restatable}[{\cite[Conjecture~21]{dwyer2018wilf}}]{corollary}{cormain} \label{cor:main}
Let $a$, $b$, $a'$, $b'$, and $m$ be integers with $1 \leq a < b \leq m$ and $1 \leq a' < b' \leq m$. Suppose that $b - a = b' - a' > 1$ and that $a + b < a' + b' \leq m+1$. Then for sufficiently large $n$ we have \[e(P_n^{m, a, b}) < e(P_n^{m, a', b'}).\]
\end{restatable}

In Section~\ref{sec:modify}, we will begin the proof of Theorem~\ref{thm:main} by defining a poset $Q_n^{m, a, b}$ such that $e(P_n^{m, a, b})$ and $e(Q_n^{m, a, b})$ are within a polynomial factor in $n$. In Section~\ref{sec:prob}, we will write $e(Q_n^{m, a, b})$ as a $(n+1)$-dimensional integral using a probabilistic interpretation of linear extensions. In Section~\ref{sec:bound}, we will compute sharp asymptotics for this integral, proving Theorem~\ref{thm:main}. In Section~\ref{sec:proof}, we will use Theorem~\ref{thm:main} to prove Corollary~\ref{cor:main} and conclude Theorem~\ref{thm:non-s}. In Section~\ref{sec:concluding}, we will discuss possible extensions of this technique and directions for further research.

\section{Modifying the Cluster Poset} \label{sec:modify}
In this section, we will define a poset $Q_n^{m, a, b}$ with the property that
\begin{equation} \label{eqn:modify} \log e(P_n^{m, a, b}) = \log e(Q_n^{m, a, b}) + O_{m, a, b}(\log n). \end{equation}
Hence, to prove Theorem~\ref{thm:main}, it will suffice to show that
\begin{equation} \label{eqn:suffices} \log e(Q_n^{m,a,b}) = (m-b+a-1)n\log n + c(m,a,b)n + O_{m,a,b}(\log n).\end{equation}

\begin{definition}
Let $a$, $b$, and $m$ be integers with $1 \leq a < b \leq m$. For any $n \geq 1$, define $Q_n^{m, a, b}$ to be $P_n^{m, a, b} \cup \{A_{0, b+1}, A_{0, b+2}, \ldots, A_{0, m}, A_{n+1, 1}, A_{n+1, 2}, \ldots, A_{n+1, a-1}\}$ with the added relations $A_{1, a} \leq A_{0, b+1} \leq A_{0, b+2} \leq \cdots \leq A_{0, m} $ and $A_{n+1, 1} \leq A_{n+1, 2} \leq \cdots \leq A_{n+1, a-1} \leq A_{n, b}$.
\end{definition}
This is a poset with cardinality $(m-1)n+m-b+a$. An example is shown for $m = 8, a = 3, b = 5$ in Figure~\ref{fig:modify-poset}. Since $P_n^{m,a,b}$ is an induced sub-poset of $Q_n^{m, a, b}$ and $|Q_n^{m,a,b}|-|P_n^{m,a,b}| = m-b+a-1$, we have
\[e(P_n^{m,a,b})\le e(Q_n^{m,a,b})\le |Q_n^{m, a, b}|^{m-b+a-1} e(P_n^{m,a,b}),\]
which implies \eqref{eqn:modify}.

\begin{figure}
\includegraphics{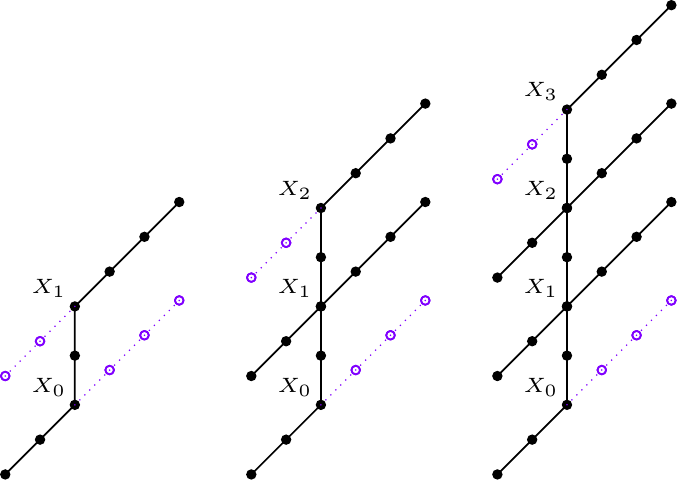}
\caption{The Hasse diagrams of the modified cluster posets $Q_1^{8,3,5}, Q_2^{8,3,5}, Q_3^{8,3,5}$. The added elements and relations are colored purple and dashed.}
\label{fig:modify-poset}
\end{figure}

\section{Probabilistic Interpretation of Linear Extensions} \label{sec:prob}
In this section, we will outline the first steps toward counting the linear extensions of $Q_n^{m, a, b}$. The key ingredient is the following probabilistic interpretation of the number of linear extensions of a poset.

Let $(P, \leq)$ be a finite poset and let $\varphi : P \to (0, 1)$ be a uniformly random function. Assuming that $\varphi$ is injective, which occurs with probability $1$, it induces a uniformly random linear order $\prec$ on $P$ given by $x \prec y$ if and only if $\varphi(x) < \varphi(y)$. The linear order $\prec$ is a linear extension of $P$ if and only if $\varphi$ is (strictly) \emph{order-preserving}; that is, if $\varphi(x) < \varphi(y)$ whenever $x < y$. Hence, we may interpret $e(P)$ probabilistically, viz. \[\frac{e(P)}{|P|!} = \Pr(\text{$\varphi$ is order-preserving}).\]

Applying this to the case $P = Q_n^{m, a, b}$, we obtain
\[\frac{e(Q_n^{m, a, b})}{((m-1)n + m-b+a)!} = \Pr(\text{$\varphi$ is order-preserving}).\]
To approximate this probability, we will first write it as a $(n+1)$-dimensional integral. Consider the elements $X_0, \ldots, X_n \in Q_n^{m, a, b}$ defined by $X_i = A_{i+1, a}$ for $0 \leq i \leq n-1$ and $X_i = A_{i, b}$ for $1 \leq i \leq n$. (These definitions coincide when $1 \leq i \leq n-1$.) Let $x_0, \ldots, x_n \in (0, 1)$. Conditioned on $\varphi(X_i) = x_i$ for $0 \leq i \leq n$, the probability that $\varphi$ is order-preserving can be computed as follows. If it is not the case that $x_0 < \cdots < x_n$, then $\varphi$ is not order-preserving. Otherwise, it is order-preserving with probability \[\prod_{i = 0}^{n} \frac{x_i^{a-1}}{(a-1)!}\frac{(1-x_{i})^{m-b}}{(m-b)!} \prod_{i=0}^{n-1}\frac{(x_{i+1} - x_i)^{b-a-1}}{(b-a-1)!}.\] Each factor $\frac{x_i^{a-1}}{(a-1)!}$ in this product is the probability that $\varphi(A_{i+1, 1}) < \cdots < \varphi(A_{i+1, a-1}) < x_i $, each factor $\frac{(1 - x_i)^{m-b}}{(m-b)!}$ is the probability that $x_i < \varphi(A_{i, b+1}) < \cdots < \varphi(A_{i, m})$, and each factor $\frac{(x_{i+1}-x_i)^{b-a-1}}{(b-a-1)!}$ is the probability that $x_i < \varphi(A_{i+1,a+1}) < \cdots < \varphi(A_{i+1,b-1}) < x_{i+1}$.

Hence
\begin{equation}\label{eqn:linext}
\frac{e(Q_n^{m, a, b})}{((m-1)n + m-b+a)!} = \int\displaylimits_{0< x_0<\cdots< x_n< 1} \prod_{i = 0}^{n} \frac{x_i^{a-1}}{(a-1)!}\frac{(1-x_{i})^{m-b}}{(m-b)!} \prod_{i=0}^{n-1}\frac{(x_{i+1} - x_i)^{b-a-1}}{(b-a-1)!}d\mathbf{x}.
\end{equation}

\section{Bounding $e(Q_n^{m, a, b})$} \label{sec:bound}
Let $I_n$ be the integral on the right-hand side of \eqref{eqn:linext}. In this section we prove Theorem~\ref{thm:main} by applying a change of variables to the integral $I_n$.

Define the function $g: [0, 1]\to [0, 1]$ by
\[g(t) = \frac{1}{\Beta\left(\frac{a - 1}{b - a}+1, \frac{m - b}{b - a}+1\right)}\int_0^t u^{\frac{a - 1}{b - a}}(1 - u)^{\frac{m - b}{b - a}}du.\]
This is a strictly increasing function, independent of $n$, with $g(0) = 0$ and $g(1) = 1$. Hence we may define its inverse $f = g^{-1} : [0, 1]\to [0, 1]$, which is also strictly increasing and independent of $n$. An example is shown for $m = 8, a = 3, b = 5$ in Figure~\ref{fig:plot}. Observe that
\[g'(t) = \frac{1}{\Beta\left(\frac{a - 1}{b - a}+1, \frac{m - b}{b - a}+1\right)}t^{\frac{a-1}{b-a}}(1-t)^{\frac{m-b}{b-a}}.\]
Using the formula $f'(t) = \frac{1}{g'(f(t))}$ yields
\begin{equation}\label{eqn:f-diffeq}f'(t)^{b-a}f(t)^{a-1}(1-f(t))^{m-b} = \Beta\left(\frac{a - 1}{b - a}+1, \frac{m - b}{b - a}+1\right)^{b-a}.\end{equation}
It follows that $f'$ is unimodal: there exists $\lambda \in [0, 1]$ such that $f'(t)$ is decreasing for $t \in [0,\lambda]$ and increasing for $t \in [\lambda, 1]$. (The constant $\lambda$ is given explicitly by $\lambda = g\left(\frac{a-1}{m-b+a-1}\right)$.) It also follows that there exists an integer $N = N(m, a, b) > 0$ such that for all $t \in (0, 1)$, we have
\begin{equation} \label{eqn:polybound}
    \frac{1}{(f'(t))^{b - a - 1}} \geq t^{N} (1-t)^N.
\end{equation} 

\begin{figure}
    \centering
    \input{plot.tex}
    \caption{A plot of $f$ for $m=8$, $a=3$, $b=5$. The marked points are \[\left(\frac{i+1}{27}, \wt{\operatorname{ht}}_\prec(X_i)\right)\] for $i = 0, \ldots, 25$, where $\prec$ is a random linear extension of $P_{25}^{8,3,5}$ generated by the Markov chain Monte Carlo algorithm of \cite{bubley1999faster}. Here $\wt{\operatorname{ht}}_\prec(X_i)$ denotes the fraction of elements of $P_{25}^{8,3,5} \setminus \{X_i\}$ that precede $X_i$ in the order $\prec$. Since the marked points are close to the plot of $f$, this figure agrees with the prediction of Conjecture~\ref{conj:concentration}.}
    \label{fig:plot}
\end{figure}

Now perform the substitution $y_i = g(x_i)$, so that $x_i = f(y_i)$. This yields
\begin{align}
I_n &= \int\displaylimits_{0< x_0<\cdots< x_n < 1} \prod_{i=0}^n \frac{x_i^{a-1}}{(a-1)!}\frac{(1-x_i)^{m-b}}{(m-b)!}\prod_{i=0}^{n-1}\frac{(x_{i+1} - x_i)^{b-a-1}}{(b-a-1)!}d\mathbf{x} \nonumber \\
&= \int\displaylimits_{0 < y_0 < \cdots < y_n < 1} \prod_{i = 0}^n \frac{f(y_i)^{a-1}}{(a-1)!}\frac{(1-f(y_i))^{m-b}}{(m-b)!}\prod_{i=0}^{n-1}\frac{(f(y_{i+1})-f(y_i))^{b-a-1}}{(b-a-1)!}\prod_{i=0}^n f'(y_i)d\mathbf{y}, \nonumber \\
&=(b-a-1)!\left(\frac{\Beta\left(\frac{a - 1}{b - a}+1, \frac{m - b}{b - a}+1\right)}{(a-1)!(m-b)!(b-a-1)!}\right)^{n + 1} \int\displaylimits_{0 < y_0 < \cdots < y_n< 1} \left(\frac{\prod_{i=0}^{n-1}(f(y_{i+1})-f(y_i))}{\prod_{i=0}^{n}f'(y_i)}\right)^{b-a-1}d\mathbf{y}, \label{eqn:lastintegral}
\end{align}
using \eqref{eqn:f-diffeq} to pass from the second to third line.

We will now bound the integrand of this last integral.
\begin{lemma}\label{lem:integrand}
Let $0 < y_0 < \cdots < y_n < 1$. Then, we have
\[\frac{f'(\lambda)}{f'(y_0)f'(y_n)}\prod_{i=1}^{n-1} (y_{i+1} - y_i)\le\frac{\prod_{i=0}^{n-1}(f(y_{i+1})-f(y_i))}{\prod_{i=0}^{n}f'(y_i)}\le\frac{1}{f'(\lambda)}\prod_{i=1}^{n-1} (y_{i+1} - y_i).\]
\end{lemma}
\begin{proof}
By the mean value theorem, for each $i$ with $0 \leq i \leq n-1$, there exists a real number $u_i$ with $y_i < u_i < y_{i+1}$ such that $f'(u_i) = \frac{f(y_{i+1}) - f(y_i)}{y_{i+1} - y_i}$. It remains to show that
\begin{equation}
\frac{f'(\lambda)}{f'(y_0) f'(y_n)} \leq \frac{\prod_{i=0}^{n-1} f'(u_i)}{\prod_{i=0}^{n} f'(y_i)} \leq \frac{1}{f'(\lambda)}. \label{eqn:altprod}
\end{equation}

Recall that $f'(t)$ is unimodal for $t \in (0, 1)$ and minimized at $t = \lambda$. Also recall that $0 < y_0 < u_0 < y_1 < u_1 < \cdots < y_{n-1} < u_{n-1} < y_n < 1$. We will prove \eqref{eqn:altprod} using only these facts. We consider four cases, depending on where $\lambda$ falls between the elements of the sequence $0, y_0, u_0, \ldots, u_{n-1}, y_n, 1$.

\begin{case}[$\lambda \leq y_0$] \label{case:case1}
Then for $0 \leq i \leq n-1$ we have $\lambda \leq y_i < u_i < y_{i+1}$, so $f'(y_i) \leq f'(u_i) \leq f'(y_{i+1})$. Hence \[\frac{\prod_{i=0}^{n-1} f'(u_i)}{\prod_{i=0}^{n} f'(y_i)} = \frac{1}{f'(y_n)} \prod_{i=0}^{n-1} \frac{f'(u_i)}{f'(y_i)} \geq \frac{1}{f'(y_n)} \geq \frac{f'(\lambda)}{f'(y_1) f'(y_n)}\] and \[\frac{\prod_{i=0}^{n-1} f'(u_i)}{\prod_{i=0}^{n} f'(y_i)} = \frac{1}{f'(y_1)} \prod_{i=0}^{n-1} \frac{f'(u_i)}{f'(y_{i+1})} \leq \frac{1}{f'(y_1)} \leq \frac{1}{f'(\lambda)}.\] The inequality \eqref{eqn:altprod} is proved.
\end{case}
\begin{case}[$y_\ell \leq \lambda \leq u_\ell$ for some $\ell$ with $0 \leq \ell \leq n-1$] \label{case:case2}
Then for $0 \leq i \leq \ell - 1$, we have $u_i < y_{i+1} \leq \lambda$, so $f'(u_i) \geq f'(y_{i+1})$. Similarly, for $\ell + 1 \leq i \leq n-1$, we have $\lambda \leq y_i < u_i$, so $f'(u_i) \geq f'(y_i)$. Hence \[\frac{\prod_{i=0}^{n-1} f'(u_i)}{\prod_{i=0}^{n} f'(y_i)} = \frac{f'(u_\ell)}{f'(y_0) f'(y_n)} \prod_{i=0}^{\ell - 1} \frac{f'(u_i)}{f'(y_{i+1})} \prod_{i=\ell + 1}^{n-1} \frac{f'(u_i)}{f'(y_i)} \geq \frac{f'(u_\ell)}{f'(y_0) f'(y_n)} \geq \frac{f'(\lambda)}{f'(y_0) f'(y_n)}.\] 

Additionally for $0 \leq i \leq \ell-1$ we have $y_i < u_i \leq \lambda$, so $f'(u_i) \leq f'(y_i)$. Similarly, for $\ell \leq i \leq n-1$ we have $\lambda \leq u_i < y_{i+1}$, so $f'(u_i) \leq f'(y_{i + 1})$. Hence \[\frac{\prod_{i=0}^{n-1} f'(u_i)}{\prod_{i=0}^{n} f'(y_i)} = \frac{1}{f'(y_\ell)}\prod_{i=0}^{\ell-1} \frac{f'(u_i)}{f'(y_i)} \prod_{i=\ell}^{n-1} \frac{f'(u_i)}{f'(y_{i+1})} \leq \frac{1}{f'(y_\ell)} \leq \frac{1}{f'(\lambda)}.\] The inequality \eqref{eqn:altprod} is proved.
\end{case}
\begin{case}[$u_\ell \leq \lambda \leq y_{\ell+1}$ for some $\ell$ with $0 \leq \ell \leq n-1$]
Then \eqref{eqn:altprod} follows from an argument similar to the one used in Case~\ref{case:case2}.
\end{case}
\begin{case}[$\lambda \geq y_n$]
Then \eqref{eqn:altprod} follows from an argument similar to the one used in Case~\ref{case:case1}. \qedhere
\end{case}
\end{proof}
From Lemma~\ref{lem:integrand} and \eqref{eqn:polybound} we obtain 
\begin{align}
& (f'(\lambda))^{b - a - 1}\int\displaylimits_{0 < y_0 < \cdots < y_n< 1} y_0^N (1-y_0)^N y_n^N (1 - y_n)^N\prod_{i=1}^{n-1} (y_{i+1} - y_i)^{b-a-1}d\mathbf{y} \nonumber \\
& \leq \int\displaylimits_{0 < y_0 < \cdots < y_n< 1} \left(\frac{\prod_{i=0}^{n-1}(f(y_{i+1})-f(y_i))}{\prod_{i=0}^{n}f'(y_i)}\right)^{b-a-1}d\mathbf{y} \nonumber \\
& \leq \frac{1}{(f'(\lambda))^{b-a-1}} \int\displaylimits_{0 < y_0 < \cdots < y_n< 1} \prod_{i=1}^{n-1} (y_{i+1} - y_i)^{b-a-1}d\mathbf{y}. \label{eqn:boundintegrals}
\end{align}
We now evaluate the the left- and right-hand sides of \eqref{eqn:boundintegrals}. The integral on the right-hand side of \eqref{eqn:boundintegrals} is a multivariate beta (or Dirichlet) integral \cite{Lejeune1839} and evaluates to 
\begin{equation}
\frac{1}{(f'(\lambda))^{b-a-1}} \frac{(\Gamma(b-a))^{n-1} (\Gamma(1))^2}{\Gamma((b-a)(n-1) + 2)} = \frac{(\Gamma(b-a))^{n}}{((b-a)n)!} n^{O_{m, a, b}(1)}. \label{eqn:rhs}
\end{equation}
The left-hand side of \eqref{eqn:boundintegrals} can be evaluated by expanding $(1 - y_0)^N$ and $y_n^N = (1 - (1 - y_n))^N$ using the binomial theorem and evaluating the result as a multivariate beta integral. This yields
\begin{align*}
    & (f'(\lambda))^{b - a - 1}\int\displaylimits_{0 < y_0 < \cdots < y_n< 1} y_0^N (1-y_0)^N y_n^N (1 - y_n)^N\prod_{i=1}^{n-1} (y_{i+1} - y_i)^{b-a-1}d\mathbf{y} \\
    &= (f'(\lambda))^{b - a - 1} \sum_{j=0}^N \sum_{k=0}^N (-1)^{j + k} \binom{N}{j} \binom{N}{k} \int\displaylimits_{0 < y_0 < \cdots < y_n< 1} y_0^{N+j} (1 - y_n)^{N+k}\prod_{i=1}^{n-1} (y_{i+1} - y_i)^{b-a-1}d\mathbf{y} \\
    &= (f'(\lambda))^{b - a - 1} \sum_{j=0}^N \sum_{k=0}^N (-1)^{j + k} \binom{N}{j} \binom{N}{k} \frac{\Gamma(N+j+1) \Gamma(N+k+1) (\Gamma(b-a))^{n-1}}{\Gamma((b-a)(n-1) + 2N + j + k + 2)}.
\end{align*}
Consider the expression $\Gamma((b-a)(n-1) + 2N + j + k + 2)$ appearing in this sum. By applying the identity $\Gamma(z) = \frac{1}{z} \Gamma(z+1)$ to this expression $2N - j - k$ times, we may write the sum as 
\begin{equation}
q_{m, a, b}(n) \frac{(\Gamma(b-a))^{n}}{\Gamma((b-a)(n-1) + 4N + 2)} = \frac{(\Gamma(b-a))^{n}}{((b-a)n)!} n^{O_{m, a, b}(1)}\label{eqn:lhs}
\end{equation}
where $q_{m, a, b}$ is a polynomial independent of $n$. By \eqref{eqn:boundintegrals}, \eqref{eqn:rhs}, and \eqref{eqn:lhs}, we have 
\begin{equation}
    \int\displaylimits_{0 < y_0 < \cdots < y_n< 1} \left(\frac{\prod_{i=0}^{n-1}(f(y_{i+1})-f(y_i))}{\prod_{i=0}^{n}f'(y_i)}\right)^{b-a-1}d\mathbf{y} = \frac{(\Gamma(b-a))^{n}}{((b-a)n)!} n^{O_{m, a, b}(1)}. \label{eqn:prodint}
\end{equation}

Combining \eqref{eqn:modify}, \eqref{eqn:linext}, \eqref{eqn:lastintegral}, and \eqref{eqn:prodint} (and taking logarithms) yields 

\begin{align*}
\log e(P_n^{m, a, b}) &= O_{m, a, b}(\log(n)) + \log(((m-1)n + m - b + a)!) \\
&+ \log\left((b-a-1)!\left(\frac{\Beta\left(\frac{a - 1}{b - a}+1, \frac{m - b}{b - a}+1\right)}{(a-1)!(m-b)!(b-a-1)!}\right)^{n + 1}\right) \\
&+ \log\left(\frac{(\Gamma(b-a))^{n}}{((b-a)n)!} n^{O_{m, a, b}(1)}\right).
\end{align*}

Theorem~\ref{thm:main} now follows from applying Stirling's approximation to the expressions $((m-1)n + m - b + a)!$ and $((b-a)n)!$. \qed

\section{Applications to strong c-Wilf equivalence} \label{sec:proof}
We now turn to the proofs of Corollary~\ref{cor:main}, Theorem~\ref{thm:non-s}, and Corollary~\ref{cor:strong}, which we now restate.
\cormain*
\thmnons*
\corstrong*
\begin{proof}[Proof of Corollary~\ref{cor:main}]
By Theorem~\ref{thm:main}, it suffices to show that $c(m, a, b) < c(m, a', b')$. Let $d = b - a = b' - a' > 1$ and define the function $A : \mathbb{R}_{\geq 0} \to \mathbb{R}$ by \[A(t) = d \log \Gamma\left(\frac{t}{d} + 1\right) - \log \Gamma(t + 1).\] We may compute \[A''(t) = \frac{1}{d} \psi^{(1)}\left(\frac{t}{d} + 1 \right) - \psi^{(1)}(t + 1)\] where $\psi^{(1)}$ is the trigamma function \cite[6.4.1]{abramowitz1964handbook}. By the series expansion for the trigamma function \cite[6.4.10]{abramowitz1964handbook}, we may write this as
\begin{align*}
A''(t) &= \frac{1}{d} \sum_{k=1}^\infty \frac{1}{\left(\frac{t}{d} + k\right)^2}  - \sum_{j=1}^\infty \frac{1}{(t + j)^2}  \\
&= \frac{1}{d} \sum_{k=1}^\infty \frac{1}{\left(\frac{t}{d} + k\right)^2}  - \sum_{k=1}^\infty \sum_{j= d (k - 1) + 1}^{d k} \frac{1}{(t + j)^2} \\
&< \frac{1}{d} \sum_{k=1}^\infty \frac{1}{\left(\frac{t}{d} + k\right)^2}  - d \sum_{k=1}^\infty \frac{1}{(t + d k)^2} \\
&= 0.
\end{align*}
Hence $A$ is strictly concave.

For any $t$, we may write
\begin{align*}
c(m, t, t+d) &= A(t-1) + A(m - t - d) \\
& - d \log \Gamma\left(\frac{m - d - 1}{d} + 2\right) + (m-1) \log(m-1) - d \log d - m + d + 1,
\end{align*}
which is a strictly concave function of $t$ that is symmetric about $t = \frac{m-d+1}{2}$. Hence it is unimodal and maximized at $t = \frac{m-d+1}{2}$. Since $a < a'$ and $a' = \frac{a' + b' - d}{2} \leq \frac{m-d+1}{2}$, we have \[c(m, a, b) = c(m, a, a+d) < c(m, a', a'+d) = c(m, a', b'),\] as desired. 
\end{proof}
\begin{proof}[Proof of Theorem~\ref{thm:non-s}]
Let $(a, b) = (\pi_1, \pi_m)$ and $(a', b') = (\tau_1, \tau_m)$. Since $\pi$ and $\tau$ are standard, we have the inequalities $a < b$ and $a + b \leq m+1$ and $a' < b'$ and $a' + b' \leq m+1$. 

As noted in Section~\ref{sec:introduction}, Dwyer and Elizalde \cite[Theorem~14]{dwyer2018wilf} showed that if $\pi, \tau \in \mathfrak{S}_m$ are non-overlapping and $\pi\sim\tau$, then the posets $P_n^\pi = P_n^{m, a, b}$ and $P_n^\tau = P_n^{m, a', b'}$ have the same number of linear extensions for all $n$. 

By Theorem~\ref{thm:main}, we have $\log e(P_n^\pi) = (m-b+a-1)n\log n + O_{m,a,b}(n)$ and $\log e(P_n^\tau) = (m-b'+a'-1)n\log n + O_{m,a',b'}(n)$, so $b-a = b'-a'$.

Assume for the sake of contradiction that $(a,b)\not= (a',b')$. Without loss of generality $a < a'$. If $b-a=b'-a' > 1$ then by Corollary~\ref{cor:main} we have that $e(P_n^\pi) < e(P_n^\tau)$ for sufficiently large $n$, which is a contradiction.

Finally, if $b-a=b'-a'=1$ we have $e(P_2^\pi) > e(P_2^\tau)$, as noted by Dwyer and Elizalde \cite[Proposition~20]{dwyer2018wilf}, which is again a contradiction. The theorem is proved.
\end{proof}
\begin{proof}[Proof of Corollary~\ref{cor:strong}]
By \cite[Theorem~8]{dwyer2018wilf}, Theorem~\ref{thm:non-s} implies the result.
\end{proof}

\section{Further remarks} \label{sec:concluding}

\subsection{Other Variational Problems}
It is clear that our method of approximating the $(n+1)$-dimensional integral $I_n$ can be extended to approximate integrals of the form
\begin{equation} \label{eqn:generalint}
\int\displaylimits_{0< x_0<\cdots< x_n< 1} \prod_{i=0}^n h(x_i)\prod_{i=0}^{n-1} (x_{i+1} - x_i)^{\beta} d\mathbf{x}
\end{equation}
for any sufficiently well-behaved function $h : [0, 1] \to \mathbb{R}_{\geq 0}$ and any $\beta \geq 0$. The method is to perform a substitution $x_i = j(y_i)$, where $j : [0, 1] \to [0, 1]$ plays the role of $f$ from Section~\ref{sec:bound}. The function $j$ can be found by solving the differential equation that $h(j(t))\cdot j'(t)^{\beta + 1}$ is constant, with the initial conditions $j(0) = 0, j(1) = 1$.

Using the calculus of variations, the function $j$ can also be described as the function that minimizes
\begin{equation} \label{eqn:variational}
\int_0^1 (\log h(j(t)) + (\beta + 1)\log j'(t))dt, 
\end{equation}
subject to $j(0) = 0, j(1) = 1$. Observe that if $x_i = j\left(\frac{i+1}{n+2}\right)$ for $0 \leq i \leq n$ for $n$ large, we have \[\log\left(\prod_{i=0}^n h(x_i)\prod_{i=0}^{n-1} (x_{i+1} - x_i)^{\beta}\right) \approx n \int_0^1 (\log h(j(t)) + \beta\log j'(t))dt - \beta n \log n,\]
and the integral appearing here is formally similar to the one appearing in \eqref{eqn:variational}. It would be interesting to explore why this is the case and provide a satisfying explanation for the extra weight of $\log j'(t)$ in \eqref{eqn:variational}.

\subsection{Conjectures and Further Questions}
Several outstanding questions about random linear extensions of the posets $P_n^{m,a,b}$ and $Q_n^{m, a, b}$ remain. Namely, we wish to demonstrate a concentration result for ``typical'' linear extensions. The following conjecture is motivated by the observation that the volume of the integral $I_n$ from \eqref{eqn:linext} appears to be tightly concentrated near the point \[(x_0, \ldots, x_n) = \left(f\left(\frac{1}{n+2}\right), \ldots, f\left(\frac{n+1}{n+2}\right)\right),\] where $(y_0, \ldots , y_n) = \left(\frac{1}{n+2}, \ldots, \frac{n+1}{n+2}\right)$ are equally spaced.
\begin{conjecture} \label{conj:concentration}
Fix integers $m, a, b$ with $1 \leq a < b \leq m$. Then, with probability tending towards $1$ as $n\rightarrow\infty$, the height of $X_i$ in a uniformly random linear extension of $P_n^{m,a,b}$ is \[|P_n^{m, a, b}| \left(f\left(\frac{i+1}{n+2}\right) + o_{m, a, b}(1)\right)\] for $0 \leq i \leq n$.
\end{conjecture}
The expected height of an element in a uniformly random linear extension of a poset is an object of independent interest \cite{winkler1982average}.

Nakamura also conjectured the following, which appears to defy the techniques we have outlined above. We remark that the natural analogue of this conjecture for ordinary Wilf-equivalence is already false for patterns of length $3$ \cite{dwyer2018wilf}.
\begin{conjecture}[{\cite[Conjecture~6]{nakamura2480computational}}]
Let $\pi, \tau \in \mathfrak{S}_m$. Then $\pi\sim\tau$ if and only if $\pi\overset{s}\sim\tau$.
\end{conjecture}
This conjecture has already been resolved in the case that $\pi$, $\tau$ are non-overlapping \cite{elizalde2012most, mendes2006permutations}.

\section*{Acknowledgements}
This research was conducted under the supervision of Joseph Gallian at the University of Minnesota Duluth REU, funded by NSF Grant 1650947 and NSA Grant H98230-18-1-0010. The authors thank Joseph Gallian, who ran the REU, brought this question to their attention, and provided helpful comments on the manuscript.
\bibliographystyle{plain}
\bibliography{main}

\end{document}